\documentclass{article}
\usepackage{latexsym}
\usepackage{amsmath,amssymb,amsthm}
\usepackage[font=small]{caption}
\usepackage[affil-it]{authblk}
\usepackage{cite}

\newtheorem{defi}{Def}
\newtheorem{coro}{Corollary}
\newtheorem{theo}{Theorem}
\newtheorem{prop}{Proposition}

\newtheorem{la}{Lemma}

\usepackage{latexsym,pgf,makeidx,dsfont,authblk,boxedminipage,tikz}
\usetikzlibrary{arrows.meta,decorations.pathmorphing,backgrounds,positioning,fit,petri,calc,graphs,graphs.standard}


\begin{document}
	
	\title{Diameter of orientations of graphs with given order and number of blocks }	
	\author[1]{P. Dankelmann}
	\author[2]{M.J. Morgan}
	\author[1]{E.J. Rivett-Carnac}
	\affil[1]{\small Department of Pure and Applied Mathematics, University of Johannesburg, Johannesburg, South Africa}
	\affil[2]{School of Mathematics, Statistics and Computer Science, University of KwaZulu-Natal, Durban, South Africa}
	
	\renewcommand\Authands{ and }
	
	\maketitle
	
	\begin{abstract}
A strong orientation of a graph $G$ is an assignment of a direction to each edge such that $G$ 
is strongly connected. The oriented diameter of $G$ is the smallest diameter among all 
strong orientations of $G$. A block of $G$ is a maximal connected subgraph of $G$ that has no cut 
vertex. A block graph is a graph in which every block is a clique. We show that every bridgeless 
graph of order $n$ containing $p$ blocks has an oriented diameter of at most 
$n-\lfloor \frac{p}{2} \rfloor$. This bound is sharp for all $n$ and $p$ with $p \geq 2$. 
As a corollary, we obtain a sharp upper bound on the oriented diameter in terms of order 
and number of 
cut vertices.  We also show that the oriented diameter of a bridgeless block graph of order 
$n$ is bounded above by $\lfloor \frac{3n}{4} \rfloor$ if $n$ is even and 
$\lfloor \frac{3(n+1)}{4} \rfloor$ if $n$ is odd.
	\end{abstract}

	\section{Introduction and terminology}
Let $G$ be a connected graph or a strong digraph. If $a$ and $b$ are vertices, then 
the \textit{distance} $d_G(a,b)$ is the minimum length, i.e., number of edges, of an $(a,b)$-path in $G$. The largest 
of the distances between any two vertices of $G$ is the \textit{diameter} of $G$, denoted by 
${\rm diam}(G)$. An \textit{orientation} of a graph $G$ is a digraph obtained from $G$ 
by assigning a direction to each edge of $G$. This paper is concerned with the 
\textit{oriented diameter} of a graph $G$, defined as the smallest diameter among all 
strong orientations of $G$, if such an orientation exists. We denote the oriented diameter 
of $G$ by $\overrightarrow{\rm diam}(G)$.\\
An edge $e$ of $G$ is a \textit{bridge} if the removal of $e$ disconnects $G$. A connected 
graph with no bridges is called \textit{bridgeless}. A well-known result, due to 
Robbins \cite{robbins}, states that every bridgeless graph has a strongly connected orientation. The study of distances in orientations of graphs was initiated by Chv\'{a}tal and Thomassen \cite{chvatal}. They showed that the oriented diameter of a bridgeless graph is at most $2 \, {\rm diam}(G)^2 + 2 \, {\rm diam}(G)$. They also constructed bridgeless graphs in which every orientation has diameter at least $\frac{1}{2} \, {\rm diam}(G)^2+{\rm diam}(G)$. The upper bound has recently been improved upon for values of ${\rm diam}(G) \geq 8$, to $1.373 \, {\rm diam}(G)^2 + 6.971 \, {\rm diam}(G)-1$ by Babu, Benson, Rajendraprasad and Vaka \cite{babu}.\\
		Bau and Dankelmann \cite{bau} presented an upper bound on the oriented diameter of bridgeless graphs in terms of order (i.e., the number of vertices) $n$ and minimum degree (i.e., the smallest of all the vertex degrees) $\delta$. Their upper bound of $\frac{11}{\delta +1}n+9$ was reduced to $\frac{7}{\delta +1}n$ by Surmacs \cite{surmacs}. More recently, Cochran \cite{cochran} also took into account the girth (i.e., the length of a shortest cycle) $g$ and showed that for any $\epsilon >0$ there is a bound of the form $(2g+\epsilon)\frac{n}{h(\delta,g)}+O(1)$, where $h(\delta,g)$ is a polynomial in $\delta$ and $g$ of degree $\lfloor \frac{g-1}{2} \rfloor$. For $g=3$ and $\epsilon<1$, this further improves upon Surmac's upper bound.\\
		A set $S$ of vertices in a graph $G$ is called a \textit{dominating set} of $G$ if every vertex in $V(G)-S$ is adjacent to some vertex in $S$. The \textit{domination number} $\gamma(G)$ of a graph $G$ is the cardinality of the smallest dominating set. Fomin, Matamala, Prisner and Rapaport  \cite{fomin1} proved that the oriented diameter can be bounded from above by $9 \gamma(G)-5$, and gave the stronger bound of $5 \gamma(G) -1$ in \cite{fomin 3}. Kurz and L\"{a}tsch \cite{kurz} further improved this upper bound to $4 \gamma(G)$ and conjectured that the minimum oriented diameter of bridgeless graphs is at most $\lceil \frac{7 \gamma(G) +1}{2}\rceil$. If true, this bound is sharp. The \textit{connected domination number} $\gamma_c(G)$ is defined as the minimum cardinality of a dominating set that induces a connected subgraph. Dankelmann, Morgan and Rivett-Carnac \cite{dankelmann} proved a sharp upper bound on the oriented diameter in terms of the connected domination number. They showed that the oriented diameter of a bridgeless graph $G$ is at most $2 \gamma_c(G) +3$ if $\gamma_c(G)$ is even, and $2 \gamma_c(G) +2$ if $\gamma_c(G)$ is odd. \\
Bounds on the oriented diameter for various graph classes have been studied. These 
include interval graphs and $2$-connected proper interval graphs \cite{fomin1, huang}, 
complete multipartite graphs \cite{koh2} and maximal outerplanar graphs \cite{chen}. 
Other results on the oriented diameter can be found, for example, in \cite{gutin,koh,kwok}. \\
A \textit{cut vertex} of a graph $G$ is a vertex whose removal, together with any incident 
edges, disconnects the graph. A \textit{block} of $G$ is a maximal connected subgraph of $G$ 
that has no cut vertex.  Harary \cite{harary} defined the \textit{block graph} $B(G)$ of a graph $G$ to be the graph in which the vertices of 
$B(G)$ are the blocks of $G$. Vertices in $B(G)$ are adjacent if the corresponding blocks in 
$G$ share a cut vertex of $G$. In this paper, we obtain a sharp upper bound on the 
oriented diameter in terms of order $n$ and number of blocks $p$. 
As a corollary to this result, we then bound the oriented diameter in terms of order and 
the number of cut vertices in a graph. A graph $G$ is a \textit{block graph}, or a \textit{clique 
tree}, if every block of $G$ is a complete graph. We also give a sharp upper bound on 
the oriented diameter of a bridgeless block graph in terms of order.\\
Also, for other distance parameters, bounds on distance parameters in terms of order and 
number of blocks or number of cut vertices are known. 
The \textit{Wiener index} is defined as the sum of distances between all (unordered) pairs 
of vertices of $G$.  Bessy, Dross, Hri\v{n}\'{a}kov\'{a}, Knor and 
\v{S}krekovski \cite{bessy, bessy-2} determined the graphs that maximise the Wiener index 
among all graphs of given order and number of blocks. 
Pandey and Patra \cite{pandey-2} gave a sharp lower bound on the Wiener index in terms of 
order and number of cut vertices. 
The \textit{eccentricity} of a vertex $v$ in $G$ is the distance from $v$ to a 
vertex farthest from $v$. The \textit{total eccentricity index} of a connected graph is the sum 
of the eccentricities of all its vertices. Pandey and Patra \cite{pandey} gave a sharp lower 
bound on the total eccentricity index given the order and number of cut vertices. 
They also provided sharp 
upper bounds on the total eccentricity index with $s$ cut vertices for $s=0,1,n-3,n-2$, and 
proposed a conjecture for $2 \leq s \leq n-4$. 
		
We will use the following notation. $G=(V,E)$ denotes a graph with vertex set $V(G)$ and edge set $E(G)$. By the \textit{order} of $G$ we mean $\vert V(G) \vert$, and usually denote this by $n$. A path $P$ from a vertex $u$ to a vertex $v$ will be referred to as a $(u,v)$-path and the \textit{length} of $P$ is denoted by $\ell (P)$. The union $G_1 \cup G_2$ of graphs $G_1=(V_1, E_1)$ and $G_2=(V_2, E_2)$ is the graph with vertex set $V_1 \cup V_2$ and edge set $E_1 \cup E_2$. For a digraph $\overrightarrow{G}$, a vertex $u$ is an \textit{in-neighbour} of $v \in V(\overrightarrow{G})$ if $\overrightarrow{uv} \in E(\overrightarrow{G})$ and an \textit{out-neighbour} if $\overrightarrow{vu} \in E(\overrightarrow{G})$. The \textit{out-eccentricity} of $v$ of $\overrightarrow{G}$ is the greatest distance from $v$ to a vertex $u \in V(\overrightarrow{G})$, and the \textit{in-eccentricity} of $v$ of 
		$\overrightarrow{G}$ is the greatest distance from a vertex $u \in V(\overrightarrow{G})$ to $v$. The \textit{eccentricity} ${\rm ecc}_{\overrightarrow{G}}(v)$ of a vertex $v$ of $\overrightarrow{G}$ 
		is the maximum of its out-eccentricity and in-eccentricity.\\
		A connected graph $G$ is said to be \textit{$2$-connected} if for every vertex $x \in V(G)$, $G-x$ is connected. Let $v$ be a cut vertex of a connected graph $G$. A \textit{branch} of $G$ at $v$ is a subgraph induced by the vertices of a component of $G-\{v\}$ together with $v$. An \textit{internal vertex} of a block of a graph $G$ is vertex which is not a cut vertex of $G$. An \textit{end block} is a block of $G$ which contains exactly one cut vertex of $G$.

\section{Results}
	
\subsection{Preliminary Results}
	
In this section we prove results that will be needed in Sections \ref{sect: oriented diameter and number of blocks} and \ref{sec: oriented diameter of block graphs }. 

	\begin{la}
	\label{la: miss at least one vertex}
Let $G$ be a bridgeless graph of order $k$, and let $x,z$ be two vertices that are not in the same block of $G$. 	
Then there exists a strong orientation $\overrightarrow{G}$ of $G$ such that 
$d_{\overrightarrow{G}} (x,z) \leq k-2$ and $d_{\overrightarrow{G}}(z,x) \leq k-2$.
	\end{la}

	\begin{proof}
Since $x$ and $z$ do not belong to the same block of $G$, there exists a cut vertex $y$ of $G$
so that $x$ and $z$ do not belong to the same component of $G-y$. Let $B_1$ and $B_2$
be the branches of $G$ at $x$ containing $x$ and $z$, respectively. 
Note that $k\geq |V(B_1)| + |V(B_2)| -1$. Since $B_1$ ($B_2$) is $2$-connected, there exist 
two edge-disjoint paths $P_1^1$ and $P_2^1$ ($P_1^2$ and $P_2^2$) from $x$ to $y$ ($y$ to $z$). 
We may assume that for $i \in \{1,2\}$ we have $\ell(P^i_1) \leq \ell(P^i_2)$. Since $P_1^i$ 
misses at least one vertex of $B_i$,  $\ell( P^i_1) \leq |V(B_i)| -2$ and 
$\ell( P^i_2) \leq |V(B_i)| -1$. \\
			We define a strong orientation $\overrightarrow{G}$  of $G$ as follows. We direct $P_1^1$ as $\overrightarrow{P_1^1}$ from $x$ to $y$ and $P_2^1$ as $\overrightarrow{P_2^1}$ from $y$ to $x$. Similarly, we direct $P_1^2$ as $\overrightarrow{P_1^2}$ from $z$ to $y$ and $P_2^2$ as $\overrightarrow{P_2^2}$ from $y$ to $z$. All remaining edges may be directed so as to extend the orientation to a strong orientation $\overrightarrow{G}$. Since $\overrightarrow{P_1^1}$, $\overrightarrow{P_2^1}$, $\overrightarrow{P_1^2}$ and $\overrightarrow{P_2^2}$ are edge-disjoint, there are no conflicts in orientation. \\
			Note that $\ell (P_1^1)+ \ell (P_2^2) \leq |V(B_1)| -2+|V(B_2)| -1 \leq k-2$. Hence, the path $\overrightarrow{P_1^1} \cup \overrightarrow{P_2^2}$ from $x$ to $z$ has length at most $k-2$, which shows that $d_{\overrightarrow{G}}(x,z) \leq k-2$. Similarly, the path $\overrightarrow{P_1^2} \cup \overrightarrow{P_2^1}$ shows that $d_{\overrightarrow{G}}(z,x) \leq k-2$. 
	\end{proof}

	\begin{la}
	\label{la: number of leaves}
		If $T$ is a tree with $p$ vertices that does not contain two adjacent vertices of degree $2$, then $T$ contains at least $\frac{p+5}{4}$ leaves. 
	\end{la}

	\begin{proof}
		Let $T$ be a tree with $p$ vertices that does not contain two adjacent vertices of degree $2$. Let $T'$ be the tree obtained from $T$ by suppressing all the vertices of degree $2$, i.e., by deleting every vertex of degree $2$ and joining its two neighbours by a new edge. Let $\vert V(T') \vert =p'$ and let $p_1$ be the number of leaves of $T'$. Clearly, $T$ also has $p_1$ leaves. Then every vertex of $T'$ that is not a leaf has degree at least $3$. That is, $p'-p_1$ vertices have degree at least $3$. Since the sum of all the vertex degrees in $T'$ is $2(p'-1)$ we get that $p_1+3(p'-p_1) \leq 2p'-2$ and thus
		\begin{equation}
		\label{eqn: order of T'}
				p' \leq 2p_1 -2.
		\end{equation}
	Since $T$ does not contain two adjacent vertices of degree $2$, we obtain $T$ from $T'$ by subdividing some edges of $T'$ once. Hence, $T$ has at most $p'+(p'-1)=2p'-1$ vertices. Since $T$ has in total $p$ vertices, and by (\ref{eqn: order of T'}), we have that
		\begin{equation*}
			\label{eqn: p<4p1-5}
				p \leq 2p'-1 \leq 2(2p_1-2)-1=4p_1-5.
		\end{equation*}
	Thus, 
	\[ p_1 \geq \frac{p+5}{4}. \qedhere \] 
	\end{proof}

\begin{la}
	\label{la: cut-block graph a tree}
	Let $G$ be a connected graph. If no cut vertex of $G$ belongs to more than two blocks, then its block graph $B(G)$ is a tree.
\end{la}

\begin{proof}
Clearly, $B(G)$ is connected. Suppose to the contrary that $B(G)$ is not a tree. 
Then $B(G)$ contains a cycle. Let blocks $B_1, B_2, \ldots B_t$ in $G$ form a cycle in $B(G)$. 
For $i=1,2,\ldots,t-1$ let $v_{i+1}$ be the cut vertex of $G$ common to $B_i$ and $B_{i+1}$, 
and let $v_1 \in V(G)$ be the cut vertex of $G$ common to $B_t$ and $B_1$. 
Since no cut vertex 
belongs to more than two blocks, the $v_i$ are distinct. There exists a $(v_1, v_t)$-path $P_1$ 
in $G$ through vertices $v_2, v_3, \ldots, v_{t-1}$ and a $(v_1, v_t)$-path $P_2$ in $G$ that 
passes through the vertices in $B_t$. Hence, $P_1$ and $P_2$ are internally disjoint and thus 
form a cycle in $G$. Clearly, the internal vertices of $P_1$ do not belong to $B_t$, while all 
vertices of $P_2$ belong to $B_t$. However, $P_1 \cup P_2$ forms a cycle in $G$, so all vertices 
in $P_1 \cup P_2$ are in  the same block. However, $v_1, v_2, \ldots , v_t$ clearly do not belong 
to the same block, a contradiction.
\end{proof}

\begin{prop}
	\label{prop: n 2p+1}
	Let $G$ be a bridgeless graph with $n$ vertices, $p$ blocks and $s$ cut vertices. Then every block of $G$ contains at least three vertices and
	\begin{equation}
	\label{eqn: n > 2p+1}
		n\geq2p+1,
	\end{equation}
	and
	\begin{equation}
	\label{eqn: s < p-1}
		n \geq 2s+3.
	\end{equation}
\end{prop}

\begin{proof}
	Clearly, a block with two vertices would contain an edge that is a bridge of $G$. Since $G$ is bridgeless, every block thus contains at least three vertices. We prove the inequality by induction on $p$. Let $G$ be a bridgeless graph of order $n$. The statement clearly holds for $p=1$. For $p \geq 2$, $G$ contains at least one cut vertex and thus, an end block, say $B$. Let $U$ be the set of internal vertices of $B$. Since every block of $G$ has at least three vertices, we have $\vert U \vert \geq 2$, and $G-U$ is a graph with $p-1$ blocks and $n-\vert U \vert$ vertices. Applying induction yields that $\vert V(G-U) \vert \geq 2(p-1)+1$, and so 
	\begin{equation*}
	\begin{split}
		n&=\vert V(G-U) \vert+\vert U \vert  \\
		&\geq 2(p-1)+1+2 \\
		&=2p+1.
	\end{split}
	\end{equation*}
Inequality (\ref{eqn: s < p-1}) follows from (\ref{eqn: n > 2p+1}) and that for $p$ blocks there are at most $p-1$ cut vertices, so $s \leq p-1$.
\end{proof}

	The remainder of results in this section are pertinent to our bound in section \ref{sec: oriented diameter of block graphs }. The following two definitions and lemma come from \cite{dankelmann}.

\begin{defi}
	Let $T$ be a not necessarily spanning subtree of a multigraph $G$. By a $T$-path we mean a path 
	in $G-E(T)$ whose ends are distinct vertices of $T$, but whose internal vertices are not. We 
	say that a $T$-path $P$ covers an edge $e$ of $T$ if $e$ is on the unique cycle contained 
	in $T\cup P$. By a $T$-cycle we mean a cycle in $G-E(T)$ that shares exactly one vertex with $T$. 
	If $P$ is a $T$-path or a $T$-cycle, then we refer to the vertices of $P$ that are not in $T$
	(that are in $T$) as the internal vertices (the ends) of $P$.  
\end{defi}	

\begin{defi}
	Let $T$ be a subtree of a multigraph $G$. Let ${\cal P}$ be a set of (not necessarily disjoint) 
	$T$-paths and $T$-cycles, and let $k \in \mathbb{N}$. We say that ${\cal P}$ is a $k$-extension 
	of $T$ in $G$ if no $T$-path or $T$-cycle in ${\cal P}$ has length greater than $k$, and 
	$T \cup \bigcup_{P \in {\cal P}} P$ is bridgeless.
\end{defi}

\begin{la}
	\label{la: tree lemma}
	Let $G$ be a multigraph and $T$ a (not necessarily spanning) subtree of $G$ of order $p$, where 
	$p\geq 1$. Let ${\cal P}$ be a $k$-extension of $T$, $k\geq 1$, and \\
	$H = T \cup \bigcup_{P \in {\cal P}} P$. Then there exists a strong orientation $D$ of 
	some submultigraph of $H$ containing $T$ such that
	\begin{enumerate}
		\renewcommand{\theenumi}{\roman{enumi}}
		\item 	for every two vertices $u,v \in V(T)$,
		\begin{equation*}
			\label{eqn: lemma bound 1}
			d_D(u,v) \leq \begin{cases}
				\frac{k+1}{2}p -1 & \textnormal{for $p$ even,} \\ 
				\frac{k+1}{2}(p-1) & \textnormal{for $p$ odd}, \\
			\end{cases}
		\end{equation*}
		\item for every two vertices $u,v \in V(D)$,
		\begin{equation*}
			\label{eqn: lemma bound 2}
			d_D(u,v) \leq \begin{cases}
				\frac{k+1}{2}p + k-2 & \textnormal{for $p$ even,} \\ 
				\frac{k+1}{2}p + \frac{k-3}{2} & \textnormal{for $p$ odd}. \\
			\end{cases}
		\end{equation*}
	\end{enumerate}
\end{la}

Plesnik \cite{plesnik}, Boesch and Tindell \cite{boesch} and Maurer \cite{maurer} proved that a complete graph on $n$ vertices, with $n\geq 3$, has an orientation of diameter $2$, unless $n=4$, in which case it has an orientation of diameter $3$. The following straightforward proposition states that $K_4$ has an orientation of diameter $3$ such that a given vertex has eccentricity $2$. 

\begin{prop}
	\label{prop: orientation Kn}
	A complete graph $K_n$, $n\geq 3$, has an orientation of diameter $2$, unless $n=4$. In this case, for every given vertex $v \in K_4$, there exists an orientation $D$ such that ${\rm ecc}_D(v)=2$.
\end{prop}

\subsection{Oriented diameter and number of blocks}
\label{sect: oriented diameter and number of blocks}

In this section we present a sharp upper bound on the oriented diameter of a bridgeless graph in terms of order and number of blocks. As a corollary, we obtain a sharp upper bound on the oriented diameter in terms of order and number of cut vertices. 

\begin{theo}
	\label{thm: bound on order ito blocks}
	For every bridgeless graph $G$ with $n$ vertices and $p$ blocks, 
	\[\overrightarrow{{\rm diam}}(G) \leq n- \bigg \lfloor \frac{p}{2} \bigg \rfloor .\]
\end{theo}

\begin{proof}
	We prove the statement by induction on $p$. The statement clearly holds for $p \leq 3$, since every strong orientation of $G$ has diameter at most $n-1$, so we assume that $p \geq 4$.\\[1mm]
	
	{\bf Case 1:} $G$ contains a cut vertex that belongs to more than two blocks.\\
	Let $v$ be such a vertex. Let $Q_1$ and $Q_2$ be two branches at $v$ and $Q_3$ the union of the remaining branches of $G$. For $i=1,2,3$, let $n_i=|V(Q_i)|$, $b_i$ be the number of blocks in $Q_i$ and let $\overrightarrow{Q_i}$ be a strong orientation of $Q_i$ of diameter $\overrightarrow{{\rm diam}}(Q_i)$. By the induction hypothesis we have ${\rm diam}(\overrightarrow{Q_i}) \leq n_i - \lfloor \frac{b_i}{2} \rfloor$. Let $D = \overrightarrow{Q_1} \cup \overrightarrow{Q_2} \cup \overrightarrow{Q_3}$. Then $D$ is a strong orientation of $G$. Since a path between two vertices of $D$ contains edges of at most two of the $\overrightarrow{Q_i}$, we have that vertices in different branches at $v$ pass through $v$. Thus
	\[{\rm diam}(D) \leq \max\{\overrightarrow{{\rm diam}}(Q_1) + \overrightarrow{{\rm diam}}(Q_2), \overrightarrow{{\rm diam}}(Q_2)+\overrightarrow{{\rm diam}}(Q_3), \overrightarrow{{\rm diam}}(Q_1)+\overrightarrow{{\rm diam}}(Q_3)\}.\]
	We may assume that the maximum above is attained by $\overrightarrow{{\rm diam}}(Q_1) + \overrightarrow{{\rm diam}}(Q_2)$. Then
	\begin{equation}
	\label{eqn: diam(Q_1)+diam(Q_2)}
			{\rm diam}(D) \leq \overrightarrow{{\rm diam}}(Q_1)+\overrightarrow{{\rm diam}}(Q_2) \leq n_1+n_2- \big( \big\lfloor \frac{b_1}{2} \big \rfloor + \big \lfloor \frac{b_2}{2} \big \rfloor \big).
	\end{equation}
	Furthermore, we have that $n=n_1+n_2+n_3-2$ and $p=b_1+b_2+b_3$. By Proposition \ref{prop: n 2p+1}, $n_3 \geq 2b_3+1$. Hence, from (\ref{eqn: diam(Q_1)+diam(Q_2)}), we have that
		\begin{equation}
		\label{eqn: diam(G)}
		\begin{split}
			{\rm diam}(D) &\leq n-n_3+2- \big( \big\lfloor \frac{b_1}{2} \big \rfloor + \big \lfloor \frac{b_2}{2} \big \rfloor \big) \\
			&\leq n-2b_3+1- \big( \big\lfloor \frac{b_1}{2} \big \rfloor + \big \lfloor \frac{b_2}{2} \big \rfloor \big).
		\end{split}
	\end{equation}

For $i=1,2$, $\lfloor \frac{b_i}{2} \rfloor \geq \frac{b_i-1}{2}$, and so (\ref{eqn: diam(G)}) becomes
	\begin{equation*}
	\begin{split}
		{\rm diam}(D) &\leq n-2b_3+1- \big(\frac{b_1-1}{2} + \frac{b_2-1}{2} \big)\\
		&=n-2b_3-\frac{p-b_3}{2}+2\\
		&=n-\frac{p}{2} - \frac{3}{2}b_3+2.
	\end{split}
\end{equation*}
For $b_3 \geq 2$ we are done, hence we assume that $b_3=1$. Then ${\rm diam}(D) \leq n-\frac{p-1}{2}$. If $p$ is odd, then we are done. If $p$ is even, then $b_1$ and $b_2$ have different parities. Without loss of generality, suppose $b_1$ is even and $b_2$ is odd. Then $\lfloor \frac{b_1}{2} \rfloor = \frac{b_1}{2}$ and $\lfloor \frac{b_2}{2} \rfloor = \frac{b_2-1}{2}$. From (\ref{eqn: diam(G)}) and substituting $b_1+b_2=p-b_3$, we get
\[{\rm diam}(D) \leq n-2b_3+1- \big(\frac{b_1}{2} + \frac{b_2-1}{2} \big)=n-\frac{p}{2},\]
and since $p$ is even, ${\rm diam}(D) \leq n- \lfloor \frac{p}{2} \rfloor$. Thus, $\overrightarrow {\rm diam}(G) \leq n- \lfloor \frac{p}{2} \rfloor$.\\[1mm]

	{\bf Case 2:} No cut vertex belongs to more than two blocks, and there exists a pair of blocks such that each have exactly two cut vertices and share a cut vertex.\\
	Let $A$ and $B$ be such a pair of blocks, where $x$ is the cut vertex of $G$ belonging to $A$, $z$ the cut vertex of $G$ belonging to $B$ and $y$ the cut vertex common to both $A$ and $B$. Let $\vert V(A) \vert = n_A$ and $\vert V(B) \vert = n_B$. Identify all the vertices in $A$ and $B$ to a new vertex $y'$. Let $G'$ be the graph obtained from $G$. This will not create any multiple edges in $G'$ since there may be vertices in $G-(V(A) \cup V(B))$ that are adjacent to $x$ or adjacent to $z$ but no vertex will be adjacent to both $x$ and $z$. We delete any loops that may arise. Notice that $\vert V(G') \vert = n-(n_A+n_B-2)$ and $G'$ contains $p-2$ blocks. \\
	By the induction hypothesis there exists a strong orientation $D'$ of $G'$ such that 
	\begin{equation}
		\label{eqn:diamG'}
	{\rm diam}(D') \leq n-(n_A+n_B-2)-\big \lfloor \frac{p-2}{2} \big \rfloor = n - \lfloor \frac{p}{2} \rfloor - (n_A+n_B-3).
	\end{equation}
	We extend the orientation $D'$ of $G'$ to a strong orientation $D$ of $G$ as follows. All edges in $D'$ that are not incident with $y'$ retain their orientations in $D$. 
Let $S_A$ ($S_B$) be the component of $G-(V(A)\cup V(B))$ containing the neighbours of 
$x$ ($z$) in $G$ that are not in $A$ ($B$). Then $y'$ is the cut vertex common to $S_A$ and 
$S_B$ in $G'$. Let $v_A \in S_A$ and $v_B \in S_B$. If an edge $v_A y'$ is oriented as 
$\overrightarrow{v_Ay'}$ ($\overrightarrow{y'v_A}$) in $G'$, then the corresponding edge, 
$v_Ax$ in $G$ is oriented as $\overrightarrow{v_Ax}$ ($\overrightarrow{xv_A}$). Similarly, if an 
edge $v_B y'$ is oriented as $\overrightarrow{v_By'}$ ($\overrightarrow{y'v_B}$) in $G'$, then 
the corresponding edge $v_Bz$ in $G$ is oriented as $\overrightarrow{v_Bz}$ ($
\overrightarrow{zv_B}$). 
We now orient the edges in $A\cup B$. 
Clearly, $A \cup B$ is a bridgeless graph of order $n_A + n_B -1$, and $x$ and $z$ are not 
in the same block of $A \cup B$. By Lemma \ref{la: miss at least one vertex}, there exists
a strong orientation $D''$ of $A\cup B$ with $d_{D''}(x,z), d_{D''}(z,x) \leq n_A+n_B-3$.
In $D$, we give every edge of $A\cup B$ an orientation as in $D''$. Then 
	\begin{equation}
		\label{eqn:distance y and z}
		d_D(x,z) \leq n_A+n_B-3 \qquad \text{and} \qquad d_D(z,x) \leq n_A+n_B-3.
	\end{equation}
	Let $a$ and $b$ be two vertices of $G$ with $d_D(a,b) = {\rm diam}(D)$. To prove the theorem in Case 2, it suffices to show that 
	\begin{equation}
		\label{eqn:distance <n-p/2}
			d_D(a,b) \leq n -\lfloor \frac{p}{2} \rfloor.
	\end{equation}
	
		{\bf Case 2.i:} $a$ and $b$ are in the same component of $G-(V(A) \cup V(B))$.\\
		In this case, $d_D(a,b)=d_{D'}(a,b)$ and so from (\ref{eqn:diamG'}) we get that
		\[d_D(a,b) \leq {\rm diam}(D') \leq n- \lfloor \frac{p}{2} \rfloor - (n_A+n_B-3),\]
		which implies (\ref{eqn:distance <n-p/2}) since $n_A, n_B \geq 3$, by Proposition \ref{prop: n 2p+1}.	\\[1mm]
		
		{\bf Case 2.ii:} $a$ and $b$ are in different components of $G-(V(A) \cup V(B))$.\\
		Without loss of generality, let $a \in S_A$ and $b \in S_B$. Then $d_D(a,x)=d_{D'}(a,y')$ and $d_D(z,b)=d_{D'}(y',b)$. Note that $d_{D'}(a,b)=d_{D'}(a,y')+d_{D'}(y', b)$. From (\ref{eqn:diamG'}) and (\ref{eqn:distance y and z}),
		\begin{equation*}
			\begin{split}
				d_D(a,b) &= d_D(a,x)+d_D(x,y)+d_D(y,z)+d_D(z,b)\\
				&\leq d_{D'}(a,b)+d_D(x,z)\\
			&\leq n - \lfloor \frac{p}{2} \rfloor - (n_A+n_B-3) +n_A+n_B-3\\
			&=n-\lfloor \frac{p}{2} \rfloor,
		\end{split}
	\end{equation*}
implying (\ref{eqn:distance <n-p/2}).\\[1mm]

			{\bf Case 2.iii:} $a,b \in V(A)\cup V(B)$.\\
			The subdigraph of $D$ induced by $V(A)\cup V(B)$ is clearly strong and has \\
			$n_A+n_B-1$ vertices. Hence, $d_D(a,b) \leq n_A+n_B-2$. Since $G'$ has $n-(n_A + n_B) +2$ vertices and $p-2$ blocks, we have, by Proposition \ref{prop: n 2p+1}, that $n-(n_A+n_B)+2 \geq 2(p-2)+1$ and so $n-(n_A+n_B) \geq 2(p-3)+1$. Hence,
			\[d_D(a,b) \leq n_A+n_B-2\leq n-2p+3,\]
			implying (\ref{eqn:distance <n-p/2}). \\[1mm]
			
			{\bf Case 2.iv:} $a\in V(A) \cup V(B)$ and $b \in G-(V(A)\cup V(B))$.\\
			Without loss of generality, let $b \in V(S_B)$. A path from $a$ to $b$ passes through $z$, i.e., $d_D(a,b)=d_D(a,z)+d_D(z,b)$. Now $d_D(a,z) \leq |V(A) \cup V(B)|-1$. Also, $d_D(z,b)=d_{D'}(y',b)$. $S_A$ contains a vertex $a'$ that is not an in-neighbour of $y'$, so $d_{D'}(a',y') \geq 2$. A shortest path from $a'$ to $b$ passes through $y'$. Since $d_{D'}(a',b) \leq {\rm diam}(D')$, then $d_{D'}(y',b) \leq {\rm diam}(D')-2$, since $a'$ is not an in-neighbour of $y'$. Hence, by (\ref{eqn:diamG'}),
				\begin{equation*}
				\begin{split}
					d_D(a,b) &\leq n_A+n_B-1+{\rm diam}(D')-2\\
					&\leq n_A+n_B-1+n-n_A-n_B+3-\lfloor \frac{p}{2} \rfloor -2\\
					&=n-\lfloor \frac{p}{2} \rfloor,
				\end{split}
			\end{equation*}
		which implies (\ref{eqn:distance <n-p/2}).\\[1mm]
			
	{\bf Case 3:} No cut vertex belongs to more than two blocks, and no two blocks each having exactly two cut vertices share a cut vertex.\\
	Let $T$ be the block graph of $G$. Then $\vert V(T) \vert = p$. Since no cut vertex of $G$ belongs to more than two blocks, by Lemma \ref{la: cut-block graph a tree}, $T$ is a tree. If $B$ is a block of $G$, then the degree of $B$ in $T$ is the number of cut vertices of $G$ in $B$. Hence, by the defining condition of Case 3, $T$ does not contain two adjacent vertices of degree two. Let $p_1$ be the number of leaves in $T$. By Lemma \ref{la: number of leaves}, $p_1  \geq \frac{p+5}{4}$. Note that the leaves of $T$ correspond to the end blocks of $G$, hence $G$ has exactly $p_1$ end blocks. Let $D$ be any strong orientation of $G$. Let $P$ be a shortest path between two vertices $a$ and $b$ with $d_D(a,b) = {\rm diam}(D)$. There are at most two end blocks that have an internal vertex that is on $P$. Hence, $P$ misses the internal vertices of at least $p_1-2$ end blocks. Since each end block has at least two internal vertices, a shortest path misses at least $2(p_1-2)$ vertices in $G$. Hence,
	\[{\rm diam}(D) = \ell(P) \leq n-2(p_1-2)-1 =n-2p_1+3. \]
	If $p_1 \geq \frac{p+6}{4}$, then
	\[{\rm diam}(D) \leq n-2\big(\frac{p+6}{4}\big)+3 =n- \frac{p}{2} \leq n-\lfloor \frac{p}{2} \rfloor.\]
	Otherwise, with $p_1 = \frac{p+5}{4}$, then
	\[ {\rm diam}(D) \leq n-2\big(\frac{p+5}{4}\big)+3 =n- \frac{p-1}{2}. \]

 Since $p_1$ is an integer, $p$ is odd, and so ${\rm diam}(D) \leq n-\lfloor \frac{p}{2} \rfloor$. This concludes the proof.
\end{proof}

We will now show that Theorem \ref{thm: bound on order ito blocks} is sharp for $p\geq 2$. For $2 \leq p \leq \frac{n-1}{2}$, consider the following construction of a graph $G_{n,p}$ with $n$ vertices and $p$ blocks that attains the bound in Theorem \ref{thm: bound on order ito blocks}. For $p=2$, $G_{n,2}$ consists of two cycles, $C_1$ and $C_2$ that share a cut vertex, $x \in V(G_{n,2})$. Let $\vert V(C_1) \vert = 3$ and $\vert V(C_2) \vert = n-2$ and let $D$ be any strong orientation of $G_{n,2}$. Since $D$ is strong, $C_1$ and $C_2$ are oriented as directed cycles. Let $c_1 \in V(C_1)$ be the out-neighbour of $x$ and $c_2 \in V(C_2)$ the in-neighbour of $x$. Then $d_D(c_1,x)=2$ and $d_D(x, c_2)=n-3$. Hence, $d_D(c_1, c_2) = 2+(n-3) = n-1$, and so ${\rm diam}(D)=n-1$. \\

For $p \geq 3$, Let $P$ be the path $a_0, a_1, \ldots a_{p-1}$. For $0 \leq i <p-1$, add a vertex $b_i$ and join it to $a_i$ and $a_{i+1}$. Thus, $a_i$, $b_i$ and $a_{i+1}$ form a $3$-cycle for each edge of $P$. Define $T_{p-1}$ to be the path $P$, together with the vertices $b_i$ and edges $a_ib_i$ and $b_ia_{i+1}$, $0 \leq i <p-1$. Let $C_{p-1}$ be a cycle of length $n-2(p-1)$ and identify a vertex of $C_{p-1}$ with $a_{p-1}$ to obtain the graph $G_{n,p}$.  Define $G_{n,p}$ to be the graph $G_{n,p}=T_{p-1} \cup C_{p-1}$. Figure \ref{pic: G5} shows a strong orientation of $G_{n,5}$.

 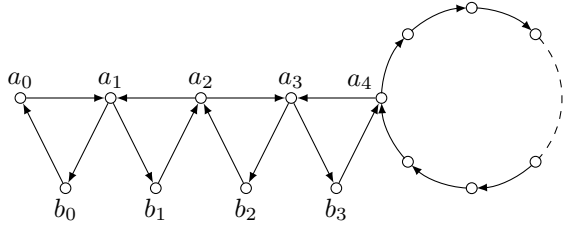
\begin{figure}[h]
	\begin{center}
		\begin{tikzpicture}
			[scale=0.6,inner sep=0.5mm, 
			vertex/.style={circle,draw}, 
			thickedge/.style={line width=2pt}] 
			
			\def\r{2} 
			\foreach \point in {(90:\r),(135:\r),(45:\r),(-45:\r),(-90:\r),(-135:\r)}
			\draw[fill=white] \point circle(.12);
			
			\node[vertex] [label=above:$a_0$] (1) at (-10,0) [fill=white] {};
			\node[vertex] [label=below:$b_0$] (2) at (-9,-2) [fill=white] {};
			\node[vertex] [label=above:$a_1$] (3) at (-8,0) [fill=white] {};
			\node[vertex]  [label=below:$b_1$] (4) at (-7,-2) [fill=white] {};
			\node[vertex] [label=above:$a_2$] (5) at (-6,0) [fill=white] {};
			\node[vertex] [label=below:$b_2$] (6) at (-5,-2) [fill=white] {};
			\node[vertex] [label=above:$a_3$] (7) at (-4,0) [fill=white] {};
			\node[vertex] [label=below:$b_3$] (8) at (-3,-2) [fill=white] {};
			\node[vertex] [label=above left:$a_4$] (9) at (-2,0) [fill=white] {};
			\draw[-latex] (2) -- (1);
			\draw[-latex] (1) -- (3);
			\draw[-latex] (3) -- (2);
			\draw[-latex] (5) -- (3);
			\draw[-latex] (3) -- (4);
			\draw[-latex] (4) -- (5);
			\draw[-latex] (5) -- (7);
			\draw[-latex] (7) -- (6);
			\draw[-latex] (6) -- (5);
			\draw[-latex] (9) -- (7);
			\draw[-latex] (7) -- (8);
			\draw[-latex] (8) -- (9);
			\draw[-latex] (177:\r) arc (177:138:\r{} and 2);
			\draw[-latex] (132:\r) arc (132:93:\r{} and 2);
			\draw[-latex] (87:\r) arc (87:48:\r{} and 2);
			\draw[-latex] (-48:\r) arc (-48:-87:\r{} and 2);
			\draw[-latex] (-93:\r) arc (-93:-132:\r{} and 2);
			\draw[-latex] (-138:\r) arc (-138:-177:\r{} and 2);
			\draw [dashed] (42:\r) arc (42:-42:\r{} and 2);
		\end{tikzpicture}
		\caption{A strong orientation of the graph $G_{n,5}$.}
		\label{pic: G5}
	\end{center}
\end{figure}

\begin{prop}
\label{prop: oriented diameter order and number of blocks sharpness}
	Let $G_{n,p}$ be the graph constructed above. Then
	\[\overrightarrow{{\rm diam}}(G_{n,p}) = n-\big\lfloor \frac{p}{2} \big\rfloor.\]
\end{prop}

\begin{proof}
	Let $D$ be a strong orientation of $G_{n,p}$ of minimum diameter. Since each $a_i, 1 \leq i \leq p-1$, is a cut vertex of $G_{n,p}$, a path from $a_1$ to $a_{p-1}$ passes through the vertices $a_1, a_2, \ldots, a_{p-1}$, in that order. Similarly, an $(a_{p-1}, a_1)$-path passes through the vertices $a_{p-1}, a_{p-2}, \ldots , a_1$, in that order. Hence,
	\begin{equation}
		\label{eqn: d(a1,ap-1)}
		d_D(a_1,a_{p-1})=d_D(a_1, a_2)+d_D(a_2, a_3) + \cdots + d_D(a_{p-2}, a_{p-1}),
	\end{equation}
and
	\begin{equation}
	\label{eqn: d(ap-1,a1)}
	d_D(a_{p-1},a_1)=d_D(a_{p-1}, a_{p-2})+d_D(a_{p-2},a_{p-3}) + \cdots + d_D(a_2,a_1).
\end{equation}
Since each edge of a $3$-cycle through $a_ia_{i+1}$ is either on the $(a_i, a_{i+1})$-path, or on the $(a_{i+1}, a_i)$-path in $D$, we have that $d_D(a_i, a_{i+1})+d_D(a_{i+1},a_i)=3$. From (\ref{eqn: d(a1,ap-1)}) and (\ref{eqn: d(ap-1,a1)}) we get that 
	\begin{equation*}
	d_D(a_1,a_{p-1})+d_D(a_{p-1},a_1)=3(p-2),
\end{equation*}
and so 
\begin{equation}
\label{eqn: max distance}
	\max \{d_D(a_1,a_{p-1}), d_D(a_{p-1}, a_1)\} \geq \bigg \lceil \frac{3(p-2)}{2} \bigg \rceil.
\end{equation}
Without loss of generality, we may assume that $d_D(a_1,a_{p-1}) \geq d_D(a_{p-1}, a_1)$. Let $q_1$ be the out-neighbour of $a_1$, which is adjacent to $a_0$, and $q_{p-1}$ the in-neighbour of $a_{p-1}$ in $V(C_{p-1})$. Note that since $\vert V(C_{p-1}) \vert = n-2(p-1)$, then $d_D(a_{p-1},q_{p-1})=n-2p+1$. We also have that $d_D(q_1,a_1)=2$. Hence,
\begin{equation*}
	\begin{split}
				\overrightarrow{{\rm diam}}(G_{n,p}) &\geq d_D(q_1,q_{p-1})= d_D(q_1,a_1)+d_D(a_1,a_{p-1})+d_D(a_{p-1},q_{p-1})\\
				&\geq 2+ \bigg \lceil \frac{3(p-2)}{2} \bigg \rceil +n-2p+1\\
				&=  n - \lfloor \frac{p}{2} \rfloor.
		\end{split}
\end{equation*}

Since $\overrightarrow{{\rm diam}}(G_{n,p}) \leq n - \lfloor \frac{p}{2} \rfloor$, by Theorem \ref{thm: bound on order ito blocks}, we have that $\overrightarrow{{\rm diam}}(G_{n,p}) = n - \lfloor \frac{p}{2} \rfloor$. So, we have shown that $G_{n,p}$ is an extremal graph, and that Theorem \ref{thm: bound on order ito blocks} is indeed sharp for all values of $p \geq 2$.
\end{proof}

It is interesting to note the comparison between our extremal graph and that of Bessy, Dross, Hri\v{n}\'{a}kov\'{a}, Knor and \v{S}krekovski \cite{bessy}, who showed that among all graphs on $n$ vertices which have $p \geq 2$ blocks, the maximum Wiener index is attained by a graph composed of two cycles joined by a path.\\

We now apply Theorem \ref{thm: bound on order ito blocks} to obtain a bound on the oriented diameter of a bridgeless graph in terms of order and number of cut vertices.

\begin{coro}
	\label{coro: cut vertices}
		For every bridgeless graph $G$, with $n$ vertices and $s$ cut vertices, 
	\[\overrightarrow{{\rm diam}}(G) \leq n- \bigg \lfloor \frac{s+1}{2} \bigg \rfloor,\]
	and this bound is sharp.
\end{coro}

\begin{proof}
	Let $G$ be a bridgeless graph of order $n$ with $s$ cut vertices. Let $p$ be the number of blocks of $G$. By Theorem \ref{thm: bound on order ito blocks}, $\overrightarrow{{\rm diam}}(G) \leq n - \lfloor \frac{p}{2} \rfloor$. However, for $p$ blocks, there are at most $p-1$ cut vertices. Hence, $s \leq p-1$, and our result follows.
\end{proof}

	The bound in Corollary \ref{coro: cut vertices} is sharp for all values of $n$ and $s$ with $1 \leq s \leq \frac{n-3}{2}$, as demonstrated by the graph $G_{n, s+1}$ constructed above.  
	
\subsection{Oriented diameter of block graphs}
\label{sec: oriented diameter of block graphs }

	We conclude the paper by determining the maximum oriented diameter of a block graph of given order. Recall that a block graph is a graph in which every block induces a complete graph. Unlike in the bound in Theorem \ref{thm: bound on order ito blocks}, here we do not prescribe the number of blocks or cut vertices.
	
		\begin{theo}
		\label{theo: block bound}
		For every bridgeless block graph $G$ with $n$ vertices,
		\begin{equation}
			\label{eqn: block bound}
			\overrightarrow{{\rm diam}}(G) \leq \begin{cases}
				\lfloor \frac{3n}{4} \rfloor& \textnormal{for $n$ even,} \\ 
				\lfloor \frac{3(n+1)}{4} \rfloor & \textnormal{for $n$ odd}. \\ 
			\end{cases}
		\end{equation}
		
	\end{theo}
	
	\begin{proof}
Let $G$ be a bridgeless block graph of order $n$. If $G$ has no cut vertices, then 
$G \cong K_n$ and so $\overrightarrow{{\rm diam}}(G) \leq 3$. Assume that $G$ has at least 
one cut vertex. Let $S= \{v_1, v_2, \ldots, v_s\}$ be the set of all cut vertices of $G$. Note 
that $G[S]$ is connected. Let $T$ be a spanning tree of $G[S]$, where $| V(T) | = s$. We 
construct a $2$-extension $\cal P$ of $T$ as follows. 

Let $B$ be a block of $G$ which is not an end block. 
First assume that $B$ has exactly two cut 
vertices, say $u$ and $v$. Since $B$ is a complete graph on at least three vertices, it 
contains at least on other vertex, say $w$. Define $P_{uv}$ to be the path $u,w,v$. 
Now assume that $B$ is a block that contains more than two cut vertices of $G$. 
For every pair $u,v$ of cut vertices that belong to $B$ but are not adjacent in $T$, let 
$P_{uv}$ be the $(u,v)$-path of length $1$. 
Let ${\cal P}$ be the set of paths defined above for the blocks with two or more cut vertices of 
$G$. Then ${\cal P}$ is a $2$-extension of $T$. Let $G'=T \cup \bigcup_{P \in {\cal P}} P$. By 
Lemma \ref{la: tree lemma} there exists a strong orientation $D'$ of a subgraph of $G'$ 
containing $T$ such that
			\begin{enumerate}
				\renewcommand{\theenumi}{\roman{enumi}}
				\item 	for every two vertices $u,v \in V(T)$,
				\begin{equation}
					\label{ineq: tree}
					d_{D'}(u,v) \leq \begin{cases}
						\frac{3}{2}s-1 & \textnormal{for $s$ even,} \\ 
						\frac{3}{2}(s-1) & \textnormal{for $s$ odd}. \\ 
					\end{cases}
				\end{equation}
				\item for every two vertices $u,v \in V(D')$,
				\begin{equation}
					\label{ineq: D'}
					d_{D'}(u,v) \leq \begin{cases}
						\frac{3}{2}s& \textnormal{for $s$ even,} \\ 
						\frac{1}{2}(3s-1) & \textnormal{for $s$ odd}. \\ 
					\end{cases}
				\end{equation}
			\end{enumerate}
			
			We now extend $D'$ to a strong orientation $D$ of $G$ such that
			\begin{equation}
				\label{orientation condition}
				\begin{minipage}{90mm}
					{every vertex not in $D$ is at distance at most $2$ to and from the nearest vertex of $T$.}
				\end{minipage}
			\end{equation}
		Let $B$ be a block of $G$. If $B$ is an end block of $G$, then it contains a unique cut vertex $w$ of $G$. By Corollary \ref{prop: orientation Kn} there exists an orientation of $B$ of diameter at most $3$ such that $w$ has eccentricity $2$.  If $B$ is not an end block, then $B$ contains two cut vertices $u, w$ of $G$, which are in $T$. For every vertex $v$ of $B$ not in $D'$ we orient the edges $uv$ and $vw$ as $\overrightarrow{uv}$ and $\overrightarrow{vw}$, respectively. All remaining edges of $B$ are oriented arbitrarily. Clearly, the resulting orientation satisfies (\ref{orientation condition}). \\
		We now bound the diameter of $D$ by showing that

	 		\begin{equation}
	 	\label{eqn: proof condition}
	 	{\rm diam}(D)  \leq \begin{cases}
	 		 \frac{3}{2}s + 3 & \textnormal{for $s$ even,} \\ 
	 	\frac{3}{2}s + \frac{5}{2}  & \textnormal{for $s$ odd}. \\ 
	 	\end{cases}
	 \end{equation}
	 	We assume that $s$ is even; the case $s$ odd is almost identical.\\[1mm]
			 {\bf Case 1:} $u,w \in V(D)$.\\[1mm]
			 From (\ref{ineq: D'}) we have $d_D(u,w) \leq \frac{3}{2}s$.\\[1mm]
			{\bf Case 2:} $u,w \in V(G) - V(D)$.\\[1mm]
			 Let $a$ and $b$ be vertices of $T$ within distance $2$ of $u$ and $w$, respectively. Then $d_D(u,w) \leq d_D(u,a) + d_D(a,b) + d_D(b,w)$. It follows from (\ref{ineq: tree}) that $d_D(u,w) \leq 2 + \frac{3}{2}s-1 + 2 = \frac{3}{2}s +3$.\\[1mm]
			{\bf Case 3:} $u \in V(D)$ and $w \in V(G)-V(D)$.\\[1mm]
			 For $w \notin V(T)$ and $a \in V(T)$, $d_D(a,w) \leq 2$ from (\ref{orientation condition}). It follows from (\ref{ineq: D'}) that $d_D(u,w) \leq d_D(u,a)+d_D(a,w) \leq \frac{3}{2}s+2$.\\[1mm]
			{\bf Case 4:} $w \in V(D)$ and $u \in V(G)-V(D)$.\\[1mm]
			 As in Case 3, $d_D(u,w) \leq \frac{3}{2}s+2$.\\[1mm]
			 
			 This proves \eqref{eqn: proof condition}. Denote the right hand side of \eqref{eqn: proof condition} by $f(s)$. Clearly, $f(s)$ is increasing in $s$. It follows from Proposition 1 that $s \leq \lfloor \frac{n-3}{2} \rfloor$. Hence,
			 
			 	 		\begin{equation*}
			 	{\rm diam}(D)  \leq  f(  \big\lfloor \frac{n-3}{2} \big\rfloor ) = \begin{cases}
			 		\frac{3}{4}n  & \textnormal{if $n \equiv 0 \pmod{4}$,} \\ 
			 		\frac{3}{4}n + \frac{1}{4} & \textnormal{if $n \equiv 1 \pmod{4}$,}\\ 
			 		 \frac{3}{4}n   - \frac{1}{2}  & \textnormal{if $n \equiv 2 \pmod{4}$,} \\ 
			 		\frac{3}{4}n + \frac{3}{4} & \textnormal{if $n \equiv 3 \pmod{4}$,}\\
			 	\end{cases}
			 \end{equation*}
		 and (\ref{eqn: block bound}) follows.
	\end{proof}

Theorem \ref{theo: block bound} is sharp for $n \geq 5$. Consider the following construction 
of a graph $G_n'$ that attains the bound. For $n$ odd, we define $G_n':=G_{n, (n-1)/2}$ as 
considered in Proposition \ref{prop: oriented diameter order and number of blocks sharpness}. 
For $n$ even, we construct $G_n'$ from $G_{n, (n-2)/2}$ by replacing the $4$-cycle with a 
complete graph $K_4$. Figure \ref{fig: G12} shows an orientation of $G_{12}'$ for $n$ even.
			
				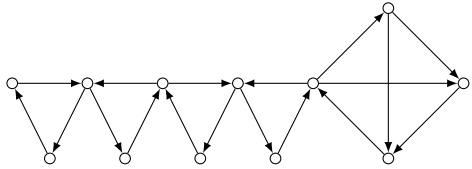
\begin{figure}[h]
				\centering
					\begin{tikzpicture}
						[scale=0.5,inner sep=0.5mm, 
						vertex/.style={circle,draw}, 
						thickedge/.style={line width=2pt}] 
						
						\node[vertex] (1) at (-10,0) [fill=white] {};
						\node[vertex] (2) at (-9,-2) [fill=white] {};
						\node[vertex]  (3) at (-8,0) [fill=white] {};
						\node[vertex]  (4) at (-7,-2) [fill=white] {};
						\node[vertex]  (5) at (-6,0) [fill=white] {};
						\node[vertex]  (6) at (-5,-2) [fill=white] {};
						\node[vertex]  (7) at (-4,0) [fill=white] {};
						\node[vertex]  (8) at (-3,-2) [fill=white] {};
						\node[vertex]  (9) at (-2,0) [fill=white] {};
						\node[vertex] (10) at (0,2) [fill=white] {};
						\node[vertex] (11) at (0,-2) [fill=white] {};
						\node[vertex] (12) at (2,0) [fill=white] {};
						\draw[-latex] (2) -- (1);
						\draw[-latex] (1) -- (3);
						\draw[-latex] (3) -- (2);
						\draw[-latex] (5) -- (3);
						\draw[-latex] (3) -- (4);
						\draw[-latex] (4) -- (5);
						\draw[-latex] (5) -- (7);
						\draw[-latex] (7) -- (6);
						\draw[-latex] (6) -- (5);
						\draw[-latex] (9) -- (7);
						\draw[-latex] (7) -- (8);
						\draw[-latex] (8) -- (9);
						\draw[-latex] (9) -- (10);
						\draw[-latex] (10) -- (11);
						\draw[-latex] (11) -- (9);
						\draw[-latex] (10) -- (12);
						\draw[-latex] (12) -- (11);
						\draw[-latex] (9) -- (12);
					\end{tikzpicture}
					\caption{The graph $G_{12}'$ with an optimal orientation.}
					\label{fig: G12}
				\end{figure}

			\begin{prop}
			\label{prop: oriented diameter block graph sharpness}
			Let $G_n'$ be the graph constructed above. Then
			\begin{equation*}
				\overrightarrow{{\rm diam}}(G_n') = \begin{cases}
					\lfloor \frac{3n}{4} \rfloor& \textnormal{for $n$ even,} \\ 
					\lfloor \frac{3(n+1)}{4} \rfloor & \textnormal{for $n$ odd}. \\ 
				\end{cases}
			\end{equation*}
		\end{prop}

\begin{proof}

For $n$ odd, we have  $G_n' =G_{n, (n-1)/2}$, and the result holds by 
Proposition \ref{prop: oriented diameter order and number of blocks sharpness}.
For $n$ even, the proof of Proposition \ref{prop: oriented diameter block graph sharpness} is 
almost identical to that of Proposition \ref{prop: oriented diameter order and number of 
blocks sharpness}, the only difference being the eccentricity of $a_{p-1}$ in the last
block, which was $3$ in the $4$-cycle and which is now at least $2$ in $K_4$. 
		\end{proof}

\end{document}